\documentclass[11pt,reqno,tbtags]{amsart}
\usepackage{amssymb}
\usepackage{url}
\usepackage{natbib}
\bibpunct[, ]{[}{]}{;}{n}{,}{,}

\numberwithin{equation}{section}

\allowdisplaybreaks

\newtheorem{theorem}{Theorem}[section]
\newtheorem{lemma}[theorem]{Lemma}

\newtheorem{corollary}[theorem]{Corollary}

\theoremstyle{definition}

\newtheorem{remark}[theorem]{Remark}

\theoremstyle{remark}

\newenvironment{romenumerate}[1][0pt]{
\addtolength{\leftmargini}{#1}\begin{enumerate}
 }{\end{enumerate}}

\newcounter{oldenumi}
{\setcounter{oldenumi}{\value{enumi}}
\begin{romenumerate} \setcounter{enumi}{\value{oldenumi}}}
{\end{romenumerate}}

\newcounter{thmenumerate}

\newcounter{romxenumerate}   

\newcounter{xenumerate}   

\newcommand{\refT}[1]{Theorem~\ref{#1}}

\newcommand{\refL}[1]{Lemma~\ref{#1}}
\newcommand{\refR}[1]{Remark~\ref{#1}}
\newcommand{\refS}[1]{Section~\ref{#1}}

\newcommand{\refand}[2]{\ref{#1} and~\ref{#2}}

\begingroup
  \divide\count255 by 60
  \multiply\count255 by -60
  \advance\count255 by \time
  \ifnum \count255 < 10 \xdef\klockan{\the\count1.0\the\count255}
  \else\xdef\klockan{\the\count1.\the\count255}\fi
\endgroup

\newcommand\nopf{\qed}   

\newcommand{\sumim}{\sum_{i=1}^m}
\newcommand{\sumin}{\sum_{i=1}^n}
\newcommand{\sumjn}{\sum_{j=1}^n}

\newcommand\set[1]{\ensuremath{\{#1\}}}

\newcommand\bigpar[1]{\bigl(#1\bigr)}
\newcommand\Bigpar[1]{\Bigl(#1\Bigr)}

\newcommand\lrpar[1]{\left(#1\right)}

\newcommand\bigabs[1]{\bigl|#1\bigr|}

\newcommand\lrabs[1]{\left|#1\right|}
\def\rompar(#1){\textup(#1\textup)}    

\def\xexp(#1){e^{#1}}

\newcommand\setn{\set{1,\dots,n}}
\newcommand\ntoo{\ensuremath{{n\to\infty}}}

\newcommand\mtoo{\ensuremath{{m\to\infty}}}

\newcommand\iid{i.i.d.\spacefactor=1000}    

\newcommand\eg{e.g.\spacefactor=1000}

\newcommand\cf{cf.\spacefactor=1000}
\newcommand{\as}{a.s.\spacefactor=1000}

\newcommand\ii{\mathrm{i}}

\newcommand{\tend}{\longrightarrow}
\newcommand\dto{\overset{\mathrm{d}}{\tend}}

\newcommand\eqd{\overset{\mathrm{d}}{=}}

\newcommand\bbR{\mathbb R}

\newcommand\bbN{\mathbb N}  

\newcommand\bbZ{\mathbb Z}

\newcounter{CC}
\newcounter{cc}
\newcommand{\cc}{\stepcounter{cc}\ccx} 
\newcommand{\ccx}{c_{\arabic{cc}}}     
\newcommand{\ccdef}[1]{\xdef#1{\ccx}}     

\newcommand\E{\operatorname{\mathbb E{}}}
\renewcommand\P{\operatorname{\mathbb P{}}}
\newcommand\Var{\operatorname{Var}}
\newcommand\Cov{\operatorname{Cov}}

\newcommand\Bi{\operatorname{Bi}}

\newcommand\gf{\varphi}

\newcommand\gs{\sigma}
\newcommand\gss{\sigma^2}
\newcommand\gth{\theta}

\newcommand\cA{\mathcal A}

\newcommand\cE{\mathcal E}

\newcommand\cG{\mathcal G}

\newcommand\cS{{\mathcal S}}

\newcommand\tH{{\tilde H}}

\newcommand\ett[1]{\boldsymbol1\set{#1}} 

\def\[#1]{[\![#1]\!]}

\newcommand\qq{^{1/2}}
\newcommand\qqc{^{3/2}}
\newcommand\qqcw{^{-3/2}}

\newcommand\qqw{^{-1/2}}

\renewcommand{\=}{:=}

\newcommand\intoooo{\int_{-\infty}^\infty}
\newcommand\oi{[0,1]}

\newcommand\dtv{d_{\mathrm{TV}}}

\newcommand\dd{\,\textup{d}}

\newcommand{\pgf}{probability generating function}

\newcommand{\chf}{characteristic function}

\newcommand\kkm{{k_1,\dots,k_m}}
\newcommand\sumkkm{\sum_{k_1+\dots+k_m=n}}
\newcommand\nnm{{n_1,\dots,n_m}}
\newcommand\NNm{{N_1,\dots,N_m}}

\newcommand\ggxn[1]{G^{(#1)}_n}
\newcommand\ggmn{\ggxn{m}}
\newcommand\ggmnq{\ggxn{m}(q)}
\newcommand\gggxn[1]{g^{(#1)}_n}
\newcommand\gggmn{\gggxn{m}}
\newcommand\gggmnq{\gggxn{m}(q)}
\newcommand\gnm{G_{n,m}}
\newcommand\gnx[1]{G_{n,#1}}
\newcommand\ellj{\ell}
\newcommand\mm{^{(m)}}
\newcommand\setm{\set{1,\dots,m}}
\newcommand\inv{\operatorname{Inv}}
\newcommand\cgnm{\cG_{n,m}}
\newcommand\cgnx[1]{\cG_{n,#1}}
\newcommand\wnm{W_{n,m}}
\newcommand\ijx{_{ij}}
\newcommand\jix{_{ji}}
\newcommand\ixij{I^*\ijx}
\newcommand\sumijn{\sum_{1\le i<j\le n}}
\newcommand\nm{_{n,m}}
\newcommand\eith{e^{\ii\theta}}
\newcommand\gfn{\gf_n}
\newcommand\bmu{\bar\mu}
\newcommand\bgs{\bar\sigma}
\newcommand\bgss{\bar\sigma^2}

\newcommand{\Takacs}{Tak\'acs}

\newcommand\urladdrx[1]{{\urladdr{\def~{{\tiny$\sim$}}#1}}}

\begin{document}
\title[Generalized Galois numbers and limit theorems]
{Generalized Galois numbers, inversions, lattice paths, Ferrers diagrams
and limit theorems}

\date{29 March, 2012} 

\author{Svante Janson}
\address{Department of Mathematics, Uppsala University, PO Box 480,
SE-751~06 Uppsala, Sweden}
\email{svante.janson@math.uu.se}
\urladdrx{http://www.math.uu.se/~svante/}

%
%

\begin{abstract}
Bliem and Kousidis recently considered a family of random variables whose
distributions are given by the generalized Galois numbers (after
normalization).
We give probabilistic interpretations of these random variables, using
inversions in random words, random lattice paths and random Ferrers diagrams,
and use these to give new proofs of limit theorems as well as some
further limit results.
\end{abstract}

\maketitle

\section{Introduction}\label{S:intro}

The \emph{homogeneous multivariate Rogers--Szeg\"o polynomial}
in $m\ge2$ variables
is defined by 
\begin{equation}
  \tH_n(t_1,\dots,t_m)
\= \sumkkm \binom{n}{\kkm}_q t_1^{k_1}\dotsm  t_m^{k_m},
\end{equation}
where $\binom n{\kkm}_q$ is the \emph{$q$-multinomial coefficient}
(or \emph{Gaussian multinomial coefficient})
\begin{equation}
  \label{qmulti}
\binom{n}{\kkm}_q :=
\frac{[n]!_q}{[k_1]!_q \cdots [k_m]!_q}
\qquad \text{for } n=k_1+\dots +k_m,
\end{equation}
where $[k]!_q:=[1]_q[2]_q \cdots [k]_q$ with $[\ell]_q:=(1-q^\ell)/(1-q)$.
Equivalently, one might consider the
\emph{inhomogeneous multivariate Rogers--Szeg\"o polynomial}
\begin{equation}
  H_n(t_1,\dots,t_{m-1})\=  \tH_n(t_1,\dots,t_{m-1},1).
\end{equation}
For these polynomials, see \citet{Rogers}, \citet{Andrews} and
\citet{Vinroot}.

We concentrate here on the special value
\begin{equation}\label{ggmn}
  \ggmn(q)=H_n(1,\dots,1)=\tH_n(1,\dots,1)
= \sumkkm \binom{n}{\kkm}_q ,
\end{equation}
studied in \citet{Vinroot} and \citet{BK}.
This is a polynomial in $q$. In the special case $m=2$,
studied in \eg{} \citet{GR}, \citet{NSW}, \citet[Chapter 7]{KacCheung} and \citet{HH}, these numbers
$\ggxn2(q)$ 
are known as \emph{Galois  numbers}, 
and the numbers  $\ggmn$ are therefore called 
\emph{generalized Galois numbers} by \cite{Vinroot} and \cite{BK}.
Note that\begin{equation}
  \ggmn(1)
= \sumkkm \binom{n}{\kkm}=m^n,
\end{equation}
by the multinomial theorem.

\citet{BK} noted that the polynomial $\ggmn(q)$ has non-negative
coefficients, and thus
\begin{equation}\label{pgf}
\gggmn(q)\=
\frac{\ggmn(q)}{\ggmn(1)}=  m^{-n}\ggmn(q)
\end{equation}
can be interpreted as the \pgf{} of a random variable $\gnm$.
We let $\cgnm$  denote the probability distribution with the \pgf{}
\eqref{pgf}, and have thus $\gnm\sim\cgnm$. (We use, following \cite{BK},
$\gnm$ for an arbitrary 
random variable with this distribution. In the next sections we will
construct specific random variables of this type.)

The purpose of the present paper is to provide some 
probabilistic interpretations of this random variable, 
see Sections \ref{Sinv}--\ref{Spaths},
and to use these
interpretations to give new, and perhaps simpler, proofs of the following
results in 
\cite{BK}. 
We use $\dto$ for convergence in distribution and (later) $\eqd$ for
equality in distribution. $N(\mu,\gss)$ is the normal distribution with mean
$\mu$ and variance $\gss$.

\begin{theorem}[\cite{BK}]
  \label{Tmean}
The random variable $\gnm$ has mean and variance
\begin{align}
  \E\gnm&=\frac{n(n-1)}4\cdot\frac{m-1}{m},\label{tmean1}\\
  \Var\gnm&=\frac{n(n-1)(2n+5)}{72}\cdot\frac{m^2-1}{m^2}.\label{tmean2}
\end{align}
\end{theorem}

\begin{theorem}[\cite{BK}]
  \label{Tlimm}
If $\mtoo$ with $n\ge 1$ fixed, then
\begin{equation}
\gnm\dto G_n,
\end{equation}
where $G_n$ is the number of inversions in a random permutation of
\setn.
\end{theorem}

\begin{theorem}[\cite{BK}]
  \label{Tlimn}
If $\ntoo$ with $m\ge 2$ fixed, then
\begin{equation}\label{tlimn1}
  \frac{\gnm-\E\gnm}{\Var(\gnm)\qq}\dto N(0,1);
\end{equation}
equivalently,
\begin{equation}\label{tlimn2}
  \frac{\gnm-\E\gnm}{n\qqc}\dto N\Bigpar{0,\frac{m^2-1}{36m^2}}.
\end{equation}
\end{theorem}

Furthermore, we can also let both $m$ and $n$ tend to infinity; we show
that there are 
no surprises in this case.

\begin{theorem}
  \label{Tlimnm}
If $m,\ntoo$, then
\begin{equation}\label{tlimnm1}
  \frac{\gnm-\E\gnm}{\Var(\gnm)\qq}\dto N(0,1);
\end{equation}
equivalently,
\begin{equation}\label{tlimnm2}
  \frac{\gnm-\E\gnm}{n\qqc}\dto N\Bigpar{0,\frac{1}{36}}.
\end{equation}
\end{theorem}

Moreover, we show a local limit theorem strengthening Theorems
\refand{Tlimn}{Tlimnm}.

\begin{theorem}\label{Tloc}
If \ntoo, then, with $\mu\nm\=\E\gnm$ and $\gss\nm\=\Var\gnm$ given by
\refT{Tmean}, 
\begin{equation}\label{tloc}
\gs\nm  \P(\gnm=k) =\frac1{\sqrt{2\pi}} e^{-(k-\mu\nm)^2/2\gss\nm}+o(1),
\end{equation}
uniformly in all $m\ge2$ and $k\in\bbZ$.

Equivalently, 
we can in \eqref{tloc} replace $\mu\nm$ and $\gss\nm$ by
the approximations
$\bmu\nm\=\frac{m-1}{4m}n^2$
and $\bgss\nm\=\frac{m^2-1}{36m^2}n^3$.
\end{theorem}

Proofs are given in Sections \ref{Spf}--\ref{Spfloc}.

\begin{remark}
  The name (generalized) Galois numbers comes from the following algebraic
  interpretation, see
\cite{GR},
\cite{Vinroot},
\cite[Chapter 7]{KacCheung}, 
\cite[Proposition 1.3.18]{StanleyI} 
which, however, not will be important in the present paper.

If $q$ is a prime power and $V$ an $n$-dimensional vector space over the
Galois field $F_q$ with $q$ elements, then it is not difficult to see that 
$\binom{n}{\kkm}_q$ is
the number of \emph{flags} $\set0\subseteq V_1\subseteq\dots\subseteq
V_m=V$, where $V_i$ is a subspace of dimension $k_1+\dots+k_i$.
Hence, $\ggmnq$ is the total number of such flags of fixed length $m$ in
$V=F_q^n$.
In particular, the Galois number $\ggxn2(q)$ is the number of subspaces of
$F_q^n$.
\end{remark}

\section{Inversions}\label{Sinv}
If $w=w_1\dotsm w_n$ is a word with letters from an ordered alphabet $\cA$,
then the number of
\emph{inversions} in $w$ is the number of pairs $(i,j)$ with $ i<j$
and $w_i> w_j$; we denote this number by $\inv(w)$.
Using the notation $\ett{\cE}$ for the indicator of an event $\cE$, we thus
have
\begin{equation}\label{inv}
  \inv(w)=\sum_{1\le i<j\le n}\ett{w_i>w_j}.
\end{equation}
With the alphabet $\cA=\set{1,\dots,m}$, 
it is well-known (and not difficult to see) that the $q$-multinomial
coefficient $\binom n{\nnm}_q$,
where $n_1+\dots+n_m=n$,  
is the generating function of the number of
inversions in words consisting of $n_1$ 1's, \dots, $n_m$ $m$'s, 
in the
sense that
if $a_\nnm(\ellj)$ is the number of such words with exactly $\ellj$
inversions, then
\begin{equation}\label{invq}
  \binom n{\nnm}_q = \sum_{\ellj=0}^\infty a_\nnm(\ellj) q^\ellj,
\end{equation}
see \cite[Theorem 3.6]{Andrews}.

Summing over all $\nnm$ with $n_1+\dots+n_m=n$, we immediately obtain the
following from \eqref{ggmn} and \eqref{invq}.

\begin{theorem}
  \label{Tinv0}
$\ggmnq$
is the generating function of the number of
inversions in words of length $n$ in the alphabet $\setm$, 
in the sense that
if $A_n\mm(\ellj)$ is the number of such words with exactly $\ellj$
inversions, then
\begin{equation}\label{tinv0}
\ggmnq = \sum_{\ellj=0}^\infty A_n\mm(\ellj) q^\ellj.
\hfill
\end{equation}
\vskip-\baselineskip
\qed
\end{theorem}

By 
the definition of the random variable $\gnm$, \eqref{tinv0} is equivalent to
\begin{equation}
  \P(\gnm=\ellj)=A_n\mm(\ellj)/n^{-m}.
\end{equation}
This can be formulated as follows, yielding our first construction of a
random variable $\gnm$.
\begin{theorem}\label{Tinv}
Let $\wnm$ be a uniformly random word   of length $n$ in the alphabet $\setm$.
Then the number of inversions $\inv(\wnm)$ has the distribution $\cgnm$.
In other words, $\gnm\eqd\inv(\wnm)$.
\nopf
\end{theorem}
We can thus choose $\gnm\=\inv(\wnm)$. (Recall that we have defined $\gnm$
to be an arbitrary random variable with the desired distribution.)

If we write  the random word $\wnm$ as $X_1\dotsm X_n$, we have
$X_1,\dots,X_n$ \iid{} (independent and identically distributed) with the
uniform distribution on \setm{}, 
and using  \eqref{inv}, \refT{Tinv} may be reformulated as follows. 

\begin{corollary}
  Let $\set{X_i}_{i=1}^\infty$ be \iid{} 
random variables, with every $X_i$ uniformly distributed on \setm{},
and let
\begin{equation}\label{V}
V\nm\=\sum_{1\le i<j\le n} \ett{X_{i}>X_{j}}.
\end{equation}
Then $V\nm\sim\cgnm$.
In other words, $\gnm\eqd V\nm$.
\nopf
\end{corollary}

Let $N_k\=\#\set{i\le n: X_i=k}$ be the number of occurences of the letter
$k$ in  the random string $\wnm=X_1\dotsm X_n$. Then $(N_1,\dots,N_m)$ has a
multinomial distribution with $\E N_k=n/m$, and it is well known that if we
keep $m$ fixed,
$n\qqw(N_k-E N_k)_{k=1}^m \dto (Z_k)_{k=1}^m$ as \ntoo,
where $Z_1,\dots,Z_m$ are jointly normal with means $\E Z_k=0$,
variances $\Var Z_k=(m-1)/m^2$ and covariances $\Cov(Z_k,Z_l)=-1/m^2$
($k\neq l$).
By \refT{Tlimn}, $V\nm\eqd\gnm$ has an asymptotic normal distribution, and
this extends to joint asymptotic normality of $V\nm$ and $N_1,\dots,N_m$.

\begin{theorem}
  \label{TV+}
For fixed $m$, as \ntoo,
\begin{equation*}
  \lrpar{\frac{V-\E V\nm}{n\qqc},\frac{N_1-E N_1}{n\qq},\dots,\frac{N_m-E N_m}{n\qq}}
\dto (Z^*,Z_1,\dots,Z_m),
\end{equation*}
where $Z^*,Z_1,\dots,Z_m$ are jointly normal with means $0$,
$\Var Z^*=(m^2-1)/36m^2$ as in \eqref{tlimn2}, $Z_1,\dots,Z_m$ have the
variances and covariances given above and $Z^*$ is independent of
$Z_1,\dots,Z_m$. 
\end{theorem}

The proof is given in \refS{Spf}.

\section{A $U$-statistic}\label{SU}

Let $\set{X_i}_{i=1}^\infty$ and $\set{Y_i}_{i=1}^\infty$ be independent
random variables, with every $X_i$ uniformly distributed on \setm{} and
every $Y_i$ uniformly distributed on \oi.
(Any common continuous distribution of $Y_i$ would yield the same result.)

Fix $n\ge1$. The values $Y_1,\dots,Y_n$ are \as{} distinct, and can thus be
ordered as $Y_{\gs(1)}<\dots<Y_{\gs(n)}$ for some (unique) permutation of
\setn.
Let $\wnm$ be the word $X_{\gs(1)}\dotsm X_{\gs(n)}$.
Since $\set{X_i}_{i=1}^n$ and $\set{Y_i}_{i=1}^n$ are independent,
$\wnm$ has the same distribution as $X_1\dotsm X_n$, and is thus a uniformly
random word in $\setm^n$.
Consequently, \refT{Tinv} yields $\inv(W\nm)\sim\cgnm$.
Moreover, since $i<j\iff Y_{\gs(i)}<Y_{\gs(j)}$,
\begin{equation*}
  \begin{split}
	\inv(W\nm)
&=\sum_{1\le i<j\le n} \ett{X_{\gs(i)}>X_{\gs(j)}}
=\sum_{i,j=1}^n \ett{X_{\gs(i)}>X_{\gs(j)} \text{ and } i<j}
\\
&=\sum_{i,j=1}^n \ett{X_{\gs(i)}>X_{\gs(j)} \text{ and } Y_{\gs(i)}<Y_{\gs(j)}}
\\
&=\sum_{k,l=1}^n \ett{X_{k}>X_{l}} \ett{Y_{k}<Y_{l}}.
  \end{split}
\end{equation*}

We have shown the following, yielding our second construction of $\gnm$.

\begin{theorem}
  Let $X_i$ and $Y_i$ be as above, and define the random variable
  \begin{equation}\label{uxy}
U\nm\=\sum_{i,j=1}^n \ett{X_{i}>X_{j}} \ett{Y_{i}<Y_{j}}.
  \end{equation}
Then $U\nm\sim \cgnm$.
In other words, $\gnm\eqd U\nm$.
\qed
\end{theorem}

Let $Z_i\=(X_i,Y_i)$; this yields a sequence of 
\iid{} random vectors taking values in
$\cS\=\setm\times\oi$.  Define the functions $h,h^*:\cS^2\to\bbR$ by 
\begin{align}
h\bigpar{(x_1,y_1),(x_2,y_2)}&\=\ett{x_{i}>x_{j}} \ett{y_{i}<y_{j}},
\label{h}
\\
h^*\bigpar{(x_1,y_1),(x_2,y_2)}&\=
h\bigpar{(x_1,y_1),(x_2,y_2)}+h\bigpar{(x_2,y_2),(x_1,y_1)}.
\label{h*}
\end{align}
Thus $h^*$ is symmetric and \eqref{uxy} can be written 
  \begin{equation}\label{U}
U\nm=\sum_{i, j=1}^n h\bigpar{Z_{i},Z_{j}}
=\sum_{1\le i< j\le n} h^*\bigpar{Z_{i},Z_{j}},
  \end{equation}
which shows that $U\nm$ is (for fixed $m$) a $U$-statistic \cite{U}.

\section{Lattice paths and Ferrers diagrams}\label{Spaths}

In this section we consider the special case $m=2$.
In this case, there is an alternative combinatorial description of the
Gaussian binomial coefficients using using lattice paths
instead of inversions, see \citet{Polya}. 
Indeed, consider lattice paths in the first quadrant, starting at the origin
and containing $n$ unit steps East or North. 
There are $2^n$ such paths, and they may be encoded by the
$2^n$ words of length $n$ with the alphabet \set{\mathsf E,\mathsf N}. 
The area under each horizontal step equals the number of previous vertical
steps,
so by summing, we see that the area under the path equals the number of
inversions in the corresponding word, where we use the ordering $\mathsf
E<\mathsf N$.

Consequently, \refT{Tinv} yields the following.
\begin{theorem}\label{Tpath}
Let $\theta(n)$ be the area under a uniformly random lattice path (of the
type above)
of length
$n$.
Then $\theta(n)\sim\cgnx2$.
In other words, $\gnx2\eqd \theta(n)$.
\end{theorem}

The random variable $\theta(n)$ was studied by
\citet{Takacs}, who found its mean and variance and proved
a central limit theorem and a local limit theorem
(our Theorems \ref{Tmean}, \ref{Tlimn} and \ref{Tloc} for $m=2$).

By symmetry, we may instead consider the area $\theta'(n)$ between the path and the
$y$-axis. This area can be regarded as a Ferrers diagram; if
the path ends at $(s_1,s_2)$, then the height (number of non-empty rows) 
$h$ and width $w$ of the Ferrers diagram satisfy 
$h\le s_2$ and $w\le s_1$, and there is a
bijection between all paths ending at $(s_1,s_2)$ and all such Ferrers diagrams.
(Note the  bijection between such Ferrers diagrams with a given area $N$ and
the partitions of $N$ into at most $s_2$ parts, each at most $s_1$;
see \cite[Theorem 3.5]{Andrews}.)

Alternatively, by adding an extra row and column, we obtain a Ferrers
diagram with height $s_2+1$  and width $s_1+1$; its right boundary consists
of a path from $(-1,0)$ to $(s_1,s_2+1)$, beginning with a horizontal step
and ending with a vertical. Moreover, there is a bijection between 
all paths ending at $(s_1,s_2)$ and all such Ferrers diagrams. We further
see that the area of this Ferrers diagram equals $\theta'+s_1+s_2+1$, where
$\theta'$ is the area between the (original) path and the $y$-axis.

The \emph{semiperimeter} of a Ferrers diagram equals its height plus width,
and we thus have obtained a bijection between all Ferrers diagram with
semiperimeter $n+2$ and all (north-east) lattice paths of length $n$.
This bijection gives a correspondence between uniformly random Ferrers
diagrams with semiperimeter $n+2$ and uniformly random lattice paths of
length $n$, yielding the following theorem.

\begin{theorem}\label{TFerrers}
  Let $A_n$ be the area 
of a uniformly random Ferrers
  diagram with semiperimeter $n+2$.
Then $A_n-n-1\sim\cgnx2$.
In other words, $\gnx2\eqd A_n-n-1$.
\end{theorem}
\begin{proof}
If $\theta'(n)$ is the area between the corresponding random lattice path
and the 
$y$-axis, then 
the arguments above show that
\begin{equation*}
  A_n=\theta'(n)+n+1\eqd\theta(n)+n+1
\end{equation*}
and the result follows by \refT{Tpath}.
\end{proof}

\begin{corollary}
The random variable $A_n$ has mean and variance
\begin{align}
  \E A_n&=
\E\gnx2+n+1=
\frac{n^2+7n+8}8 , \label{amean1}\\
  \Var A_n&=
\Var\gnx2
=\frac{n(n-1)(2n+5)}{96}.\label{amean2}
\end{align}
\end{corollary}
\begin{proof}
By Theorems \refand{TFerrers}{Tmean}.  
\end{proof}

\refT{Tlimn} yields the central limit theorem
\begin{equation}
  \frac{ A_n-\E A_n}{\Var( A_n)\qq}\dto N(0,1);
\end{equation}
by \eqref{amean1}--\eqref{amean2}, this can also be written as 
\begin{equation}
  \frac{ A_n-n^2/8}{n\qqc}\dto N\Bigpar{0,\frac1{48}},
\end{equation}
which was proved by other methods by \citet{Schwerdtfeger}.
Furthermore, \citet{Schwerdtfeger} showed that if $H_n$ is the height of the
Ferrers diagram, then there is joint convergence  of the normalised variables 
\begin{equation}
\lrpar{\frac{A_n-n^2/8}{\sqrt{n^3/48}} , \frac{H_n-n/2}{\sqrt{n/4}} }
\dto (\zeta_1,\zeta_2),
\end{equation}
where $\zeta_1,\zeta_2$ are independent standard normal variables.
The asymptotic normality of $H_n$ is immediate, since $H_n-1$ is the
$y$-coordinate of the endpoint of the corresponding lattice path, and thus
$H_n-1$ has the binomial distribution $\Bi(n,1/2)$. The joint convergence
follows  by \refT{TV+}.

\section{Proofs of Theorems \ref{Tmean}--\ref{Tlimnm} and \ref{TV+}}\label{Spf}

We will base most of the proofs on the representation in
\eqref{uxy}--\eqref{U}. (It is also possible to use \eqref{V}, 
see \refR{Rvdecomp} and the proof of \refT{TV+};
\eqref{V}  is
simpler in some ways, but we prefer the symmetry in \eqref{uxy}--\eqref{U}.)

We use the notations, with $Z$, $h$, $h^*$ as in \refS{SU}, see
\eqref{h}--\eqref{h*}, 
\begin{align}
  I\ijx&\=h(Z_i,Z_j)=\ett{X_i>X_j}\ett{Y_i<Y_j},
\label{iij}\\
\ixij&\=h^*(Z_i,Z_j)=I\ijx+I_{ji}. \label{iij*}
\end{align}
Thus \eqref{U} can be written
\begin{equation}\label{master}
  \gnm\eqd U\nm =\sumijn \ixij.
\end{equation}

\begin{proof}[Proof of \refT{Tmean}]
  By symmetry and the independence of $I\ijx$ and $I_{kl}$ when \set{i,j}
  and \set{k,l} are disjoint,
\eqref{master} implies
\begin{align}
  \E \gnm& = \binom n2 \E I^*_{12}=n(n-1) \E I_{12},\label{e1}
\\
\Var \gnm& = \binom n2 \Var I^*_{12}
+n(n-1)(n-2) \Cov\bigpar{I^*_{12}, I^*_{13}}. \label{v1}
\end{align}
Clearly,
\begin{equation}
\E I\ijx =\P(X_i>X_j)\P(Y_i<Y_j)=\frac{\binom m2}{m^2}\cdot \frac12  
=\frac{m-1}{4m}
\end{equation}
and 
\begin{equation}\label{qk}
  \E\ixij =2\E I\ijx =\frac{m-1}{2m}=\frac12-\frac1{2m};
\end{equation}
any of these yields \eqref{tmean1} by \eqref{e1}.

Since $\ixij$ is $0/1$-valued, it follows from \eqref{qk} also that
\begin{equation}\label{v2}
\Var\ixij = \E\ixij(1-\E \ixij)=\frac14\Bigpar{1-\frac1{m^2}}.
\end{equation}
Furthermore, again using symmetry,
\begin{equation*}
  \begin{split}
 & \E\bigpar{I^*_{12}I^*_{13}}
=   2\E\bigpar{I_{12}I_{13}} +   2\E\bigpar{I_{21}I_{13}}
\\&
=2\P\bigpar{X_1>X_2,X_3}\P\bigpar{Y_1<Y_2,Y_3}
+
2\P\bigpar{X_2>X_1>X_3}\P\bigpar{Y_2<Y_1<Y_3}
\\&
=2\frac{\sumim(i-1)^2}{m^3}\cdot\frac13 + 2\frac{\binom m3}{m^3}\cdot\frac16
=\frac{m(m-1)(2m-1)}{9m^3}+\frac{m(m-1)(m-2)}{18m^3}
\\&
=\frac{(m-1)(5m-4)}{18m^2}
  \end{split}
\end{equation*}
and hence
\begin{equation}\label{v3}
  \begin{split}
  \Cov\bigpar{I^*_{12},I^*_{13}}
&= \E\bigpar{I^*_{12}I^*_{13}} - \E\bigpar{I^*_{12}}^2
=\frac{(m-1)(5m-4)}{18m^2} - \frac{(m-1)^2}{4m^2}
\\&
=\frac{(m-1)(m+1)}{36m^2}.
  \end{split}
\end{equation}
The variance formula \eqref{tmean2} follows from \eqref{v1}, \eqref{v2} and
\eqref{v3}. 
\end{proof}

\begin{proof}[Proof of Theorem \ref{Tlimm}]
Consider the random word $\wnm=X_1\dotsm X_n$ in \refT{Tinv}.
If we condition on the letters $X_1,\dots,X_m$ being distinct, then the
number of inversions $\inv(\wnm)$ has the same distribution as the number
$G_n$ of inversions in a random permutation.
Hence, for any set $A\subset\bbN$,
\begin{equation*}
  \P\bigpar{\inv(\wnm)\in A\mid X_1,\dots, X_n \text{ distinct}}
=\P(G_n\in A)
\end{equation*}
and thus
\begin{multline*}
\lrabs{\P\bigpar{\inv(\wnm)\in A}-\P(G_n\in A)}
\le\P(X_1,\dots, X_n \text{ not distinct})
\\
\le\binom n2\P(X_1=X_2)=\frac{\binom n2}m\to0	  
\end{multline*}
as \mtoo, and thus $\gnm\eqd\inv(\wnm)\dto G_n$.
\end{proof}

\begin{remark}
  We have actually proved that the total variation distance
  $\dtv(\gnm,G_n)\le\binom n2/m$. Moreover, the bound can be improved to 
$$
\dtv(\gnm,G_n)\le\P(X_1,\dots, X_n \text{ not distinct})=1-(m)_n/m^n,
$$ 
where $(m)_n\=m!/(m-n)!$.
\end{remark}

\begin{proof}[Proof of Theorems \ref{Tlimn} and \ref{Tlimnm}]
The two versions in each theorem are equivalent by \eqref{tmean2}, so it
suffices to prove, for example, \eqref{tlimn2} and \eqref{tlimnm2}.

  The central limit theorem \refT{Tlimn} follows immediately from
  Hoeffding's central limit theorem for $U$-statistics \cite{U}
  without any further calculations.
Moreover, we shall see that
the decomposition method used by Hoeffding \cite{U}
yields also \refT{Tlimnm}; we therefore do the decomposition explicitly.

The idea is to decompose each term $\ixij$ as 
\begin{equation}
  \label{h1}
\ixij=\mu+\xi_i+\xi_j+\eta\ijx,
\end{equation}
where $\mu\=\E\ixij$,
\begin{equation}
  \xi_i\=\E\bigpar{\ixij-\mu\mid Z_i}=\E\bigpar{\ixij\mid X_i,Y_i}-\mu
\end{equation}
and $\eta\ijx$ is defined by \eqref{h1}.
Then the random variables $\xi_i$ ($1\le i\le n$)  and
$\eta\ijx$ ($1\le i<j\le n$) have mean 0 and are orthogonal (in $L^2$), so
they are uncorrelated. In particular,
\begin{equation}\label{qa}
  1\ge \Var\ixij =\Var\xi_i+\Var\xi_j+\Var\eta\ijx.
\end{equation}
Moreover, $\xi_i=g(Z_i)$ for some function $g$, and
thus the variables
$\xi_i$ are i.i.d.

By summing \eqref{h1}, we obtain by \eqref{master}
a corresponding decomposition of $U\nm$:
\begin{equation}\label{u1}
  U\nm=\binom n2\mu + (n-1)\sumin\xi_i+\sumijn\eta\ijx.
\end{equation}
Hence,
\begin{equation}\label{u2}
\frac{ U\nm-\E U\nm}{n\qqc}=\frac{n-1}{n} n\qqw\sumin\xi_i+n\qqcw R,
\end{equation}
where $R\=\sumijn\eta\ijx$. 
Since the variables $\eta\ijx$ are uncorrelated, and $\Var\eta\ijx\le1$ by
\eqref{qa}, we have 
\begin{equation}\label{ER2}
\E R^2=  \Var R =\sumijn\Var \eta\ijx 
\le \binom n2 \le n^2,
\end{equation}
and thus $\E(n\qqcw R)^2\to0$. 
Hence, the last term in \eqref{u2} is a small
remainder term that can be ignored when \ntoo.
Furthermore, the decomposition \eqref{u1} yields the variance decomposition
\begin{equation}\label{ql}
  \begin{split}
\Var  U\nm&= (n-1)^2\sumin\Var\xi_i+\sumijn\Var\eta\ijx
\\&
=n(n-1)^2\Var\xi_1+\binom n2 \Var\eta_{12}
\\&
\sim n^3\Var\xi_1	
  \end{split}
\raisetag\baselineskip
\end{equation}
as \ntoo, and thus by \eqref{tmean2},
\begin{equation}\label{vxi}
  \Var\xi_1=\frac{1}{36}\Bigpar{1-\frac1{m^2}}.
\end{equation}

For fixed $m$ (\refT{Tlimn}), the standard central limit theorem for sums of
\iid{} random variables now shows that 
\begin{equation}
  \frac{\sumin\xi_i}{n\qq}\dto
 N\Bigpar{0,\frac{m^2-1}{36m^2}},
\end{equation}
and thus \eqref{tlimn2} follows from \eqref{u2}.

For $m\to\infty$ (\refT{Tlimnm}), we have $\Var\xi_1\to1/36$ by \eqref{vxi};
moreover, the random variables $\xi_i$ are uniformly bounded (by 1), and
thus the central limit theorem with \eg{} Lyapounov's condition
\cite[Theorem 7.2.2]{Gut} applies and shows that
\begin{equation}
  \frac{\sumin\xi_i}{n\qq}\dto
 N\Bigpar{0,\frac{1}{36}},
\end{equation}
and thus \eqref{tlimnm2} follows from \eqref{u2}.
\end{proof}

\begin{remark}
  It is interesting to do the decomposition \eqref{h1} explicitly.
Using the centred variables 
\begin{align}
X'_i&\=X_i-\E X_i=X_i-\frac{m+1}2, \label{x'}
\\
Y'_i&\=Y_i-\E Y_i= Y_i-\frac12,  
\end{align}
we have by \eqref{iij}
\begin{align}
  \E(I\ijx\mid X_i,Y_i)
&=
\frac{X_i-1}m (1-Y_i)
=
\frac{X'_i+(m-1)/2}m \Bigpar{\frac12-Y'_i}, \label{chu}
\\
  \E(I\jix\mid X_i,Y_i)
&=
\frac{m-X_i}m Y_i
=
\frac{(m-1)/2-X'_i}m \Bigpar{Y'_i+\frac12},
\end{align}
{and thus}, using \eqref{iij*} and \eqref{qk},
\begin{equation}
  \label{rxi}
\xi_i
\=
  \E(I^*\ijx\mid X_i,Y_i)  -   \E I^*\ijx
=
-\frac{2}m X'_iY'_i.
\end{equation}
Hence, the decomposition is
\begin{equation}
I^*\ijx
=
\frac{m-1}{2m}-\frac{2}m X'_iY'_i-\frac{2}m X'_jY'_j +\eta\ijx
\end{equation}
and
\begin{equation}
U\nm=
\binom n2 \frac{m-1}{2m}-\frac{2(n-1)}m \sumin X'_iY'_i + R.
\end{equation}

Note also that \eqref{vxi} follows from \eqref{rxi}, and then 
\eqref{qa} yields, using \eqref{v2},
\begin{equation}\label{veta}
  \Var\eta\ijx=\Var\ixij-2\Var\xi_i
=\frac14\Bigpar{1-\frac1{m^2}}-\frac2{36}\Bigpar{1-\frac1{m^2}}
=\frac7{36}\Bigpar{1-\frac1{m^2}},
\end{equation}
which together with \eqref{vxi} and \eqref{ql} yield another proof of
\eqref{tmean2}.
\end{remark}

\begin{remark}\label{Rvdecomp}
  It is also interesting to do the corresponding orthogonal decomposition of
  $V\nm$ in \eqref{V}.
We have, similarly to \eqref{h1},
\begin{equation}
  \label{vx}
\ett{X_i>X_j}=\mu'+\xi'_i+\xi''_j+\eta'\ijx,
\end{equation}
where $\mu'\=\P(X_i>X_j)=\frac{m-1}{2m}$,
 and, with $X_i'$ as in \eqref{x'},
  \begin{align}
  \xi'_i&\=\P\bigpar{X_i>X_j\mid X_i}-\mu'=\frac{X'_i}{m},
\\	
  \xi''_j&\=\P\bigpar{X_i>X_j\mid X_j}-\mu'=-\frac{X'_j}{m},
  \end{align}
and $\eta'\ijx$ is defined by \eqref{vx}.
Summing we get, 
\begin{equation}\label{vv}
  \begin{split}
	  V\nm&=\E V\nm + \sumin(n-i)\xi_i'+\sumjn(j-1)\xi''_j+\sumijn\eta'\ijx
\\&
 =\E V\nm + \frac1m \sumin(n+1-2i)X_i'+\sumijn\eta'\ijx.
  \end{split}
\end{equation}

Straightforward calculations show that
\begin{align}
   \Var X_i'&= \frac1{12}(m^2-1), \\
\Var(\ett{X_i>X_j})&=\frac14\Bigpar{1-\frac1{m^2}},
\end{align}
and, by \eqref{vx},
\begin{equation}\label{vvv}
   \Var\eta'\ijx =\Var(\ett{X_i>X_j})-\Var\xi'_i-\Var\xi''_j
=\frac1{12}\Bigpar{1-\frac1{m^2}}.
\end{equation}
Hence, \eqref{vv} yields 
\begin{equation}
  \begin{split}
  \Var V\nm&=
\frac1{m^2} \sumin(n+1-2i)^2\Var X_i'+\sumijn\Var\eta'\ijx \\
&=\frac{n(n-1)(n+1)}{36}\Bigpar{1-\frac1{m^2}}
 +\frac{n(n-1)}{24}\Bigpar{1-\frac1{m^2}},
  \end{split}
\end{equation}
which gives yet another proof of
\eqref{tmean2}.

We can also prove Theorems \ref{Tlimn} and \ref{Tlimnm} using \eqref{vv}
instead of \eqref{u1};
again the final sum can be ignored since, using \eqref{vvv} and the fact
that the $\eta'\ijx$ are 
uncorrelated,
\begin{equation}\label{vvx}
  \Var\Bigpar{n\qqcw\sum_{i<j}\eta'\ijx}
= n^{-3}\binom n2 \frac1{12}\Bigpar{1-\frac1{m^2}}
<\frac1{24n}\to0
\end{equation}
as \ntoo, \cf{} \eqref{ER2}.
The summands in 
$ \sumin(n+1-2i)X_i'$
are not identically distributed,
but that does not matter since Lyapounov's condition holds. 
See \cite[Corollary 11.20]{SJIII}
for a general limit theorem for asymmetric sums like
\eqref{V}, and note that the argument in \refS{SU} is an instance of a
general method to convert such sums into (symmetric) $U$-statistics by
introducing the auxiliary variables $Y_i$, see \cite[Remark 11.21]{SJIII}.

In the case $m=2$, one can check that $\eta'\ijx=-X'_iX'_j$
and thus 
\begin{equation}
\sumijn\eta'\ijx=-\frac12\Bigpar{\sumin X'_i}^2+\frac n2,  
\end{equation}
which shows that the decomposition \eqref{vv} then is essentially the same as 
the decomposition used by \citet{Takacs}.
\end{remark}

\begin{proof}[Proof of Theorem \ref{TV+}]
We use the decomposition \eqref{vv} of $V\nm$, and $N_k=\sumin\ett{X_i=k}$.
The result follows by the central limit theorem with Lyapounov's condition
applied to the random 
vector 
\begin{multline*}
\lrpar{\frac{\sumin(n+1-2i)X_i'}{n\qqc},\,
 \frac{N_1-\E N_1}{n\qq},\dots,\frac{N_m-\E N_m}{n\qq}}
\\=
\sumin\lrpar{\frac{(n+1-2i)X_i'}{n\qqc},
\frac{\ett{X_i=1}-1/m}{n\qq},\dots,
\frac{\ett{X_i=m}-1/m}{n\qq}},
\end{multline*}
together with \eqref{vv} and \eqref{vvx}; the variances and covariances are
easily computed, noting that
$\Cov\bigpar{\sumin(n+1-2i)X_i',\;\sumin\ett{X_i=k}}=0$ for each $k$ since 
$\sumin(n+1-2i)=0$.
(This vector-valued central limit theorem follows, as is well-known, 
from the real-valued
version
\cite[Theorem 7.2.2]{Gut}
by the Cram\'er--Wold device \cite[Theorem 5.10.5]{Gut}.)
\end{proof}

\section{Proof of \refT{Tloc}}\label{Spfloc}

To prove the local limit theorem \refT{Tloc}, we need estimates of the
\pgf{} $\gggmnq=m^{-n}\ggmnq$ for $q=\eith$ on the unit circle. We derive these
estimates from the corresponding estimates of $\binom{n}{\nnm}_q$ in
\cite{SJ239} rather than from scratch. (We do not know whether the estimates
below are the best possible.) 

Consider a random word $\wnm$ as in \refS{Sinv}, let again $N_1,\dots,N_m$
be the number of occurrences of the different letters, and let
$N^*\=\max_{k\le n}N_k$ and $N_*\=n-N^*$.
Similarly, for given $\nnm$ with $n_1+\dots+n_m=n$, let
$n^*\=\max_{k\le n}n_k$ and $n_*\=n-n^*$; let further
\begin{equation*}
  F_\nnm(q)\=\binom{n}{\nnm}_q\Big/\binom{n}{\nnm}
\end{equation*} 
be the \pgf{} of the number of inversions in a random word  consisting of
$n_1$ 1's, \dots, $n_m$ $m$'s, \cf{} \eqref{invq}.
Thus $F_\nnm(q)$ is the \pgf{} of $V\nm=\inv(W\nm)$ conditioned on
$N_k=n_k$, $k=1,\dots,m$.

\begin{lemma}
  \label{LGW}
There exists $c>0$ such that for all $m\ge2$, $n\ge2$ and real
$\gth\in[-\pi,\pi]$, 
\begin{equation}\label{lgw}
  \bigabs{\gggmn(\eith)}
\le
\begin{cases}
  e^{-cn^3\gth^2},&0\le|\gth|\le1/n,\\
  e^{-cn},&1/n\le|\gth|\le\pi.
\end{cases}
\end{equation}
\end{lemma}

\begin{proof}
We assume in the proof for simplicity that $n$ is large enough; this case is
enough for our application in \refT{Tloc}.
It is easy (but not very interesting) to complete the proof by verifying the
estimates \eqref{lgw}
for each fixed $n\ge2$ and some $c$ (that now might depend on $n$); we omit 
the details but mention that the case when $m$ is large follows using
\refT{Tlimm}. 
We let $c_1,c_2,\dots$ denote some positive constants whose values are not
important. 

  By \cite[Lemma 4.1]{SJ239} there exists $\tau\in(0,1)$ such that if
  $|\gth|\le\tau/n$, then for any $\nnm$ with $n_1+\dots+n_m=n$,
  \begin{equation*}
\bigabs{F_\nnm(\eith)}	
\le e^{-\gss\gth^2/4},
  \end{equation*}
where $\gss$ depends on $\nnm$ and
by \cite[Lemma 3.1]{SJ239} $\gss\ge n^2n_*/36$.
Furthermore,
  by \cite[Lemma 4.4]{SJ239} there exists $\cc>0$ such that 
if $\tau/n\le|\gth|\le\pi$, then
  \begin{equation*}
\bigabs{F_\nnm(\eith)}	
\le e^{-\ccx n_*}.  
  \end{equation*}
Hence, if $n^*\le 3n/4$ so that $n_*\ge n/4$ we have
the estimates
  \begin{equation}\label{fb}
\bigabs{F_\nnm(\eith)}	
\le e^{-\cc n^3\gth^2},
\qquad |\gth|\le\tau/n,  \ccdef\ccf
  \end{equation}
and
  \begin{equation}\label{fc}
\bigabs{F_\nnm(\eith)}	
\le e^{-\cc n},
\qquad \tau/n\le|\gth|\le\pi.  \ccdef\ccfc
  \end{equation}

We return to our string $\wnm$ with random numbers $\NNm$
of different letters. 
We can, for any $m\ge2$, 
partition \set{1,\dots,m} into three sets with at most
$m/2$ elements each, and thus
\begin{equation}\label{ch}
  \P(N^*>3n/4)\le 3 \P\bigpar{\Bi(n,1/2)>3n/4}
\le 3e^{-\cc n} \ccdef\ccch
\end{equation}
by Chernoff's inequality, see \eg{} \cite[Theorem 2.1]{JLR}.

When $|\gth|\le\tau/n$, 
which implies $n^3\gth^2=O(n)$,
we obtain by \eqref{fb} and \eqref{ch},
\begin{equation}\label{jesper}
  \begin{split}
\bigabs{\gggmn(\eith)}
&=\bigabs{\E e^{\ii\gth V\nm}}	
\\&
=\Bigl|\E \bigpar{e^{\ii\gth V\nm}\mid N^*\le3n/4}\P(N^*\le3n/4)
\\&\qquad
+\E \bigpar{e^{\ii\gth V\nm}\mid N^*>3n/4}\P(N^*>3n/4)\Bigr|
\\&
\le e^{-\ccf n^3\gth^2}\P(N^*\le3n/4)+\P(N^*>3n/4)
\\&
\le e^{-\ccf n^3\gth^2}+3 e^{-\ccch n}
\\&
\le 4e^{-\cc n^3\gth^2}.
  \end{split}
\raisetag\baselineskip
\end{equation}
This verifies \eqref{lgw} with some $c>0$ for $\cc n\qqcw\le|\gth|\le\tau/n$.

For $|\gth|<\ccx n\qqcw$, we first note that $\P(N^*\le 3n/4) \ge \cc>0$ for
all $m,n\ge2$; \ccdef\cclgw
this holds for all large $n$ by \eqref{ch} (and is easily seen for each
fixed $n$).
Hence, by the calculations in \eqref{jesper},
\begin{equation*}
  \begin{split}
 1-\bigabs{\gggmn(\eith)}
&\ge
1-\P(N^*>3n/4)
- \P(N^*\le3n/4)e^{-\ccf n^3\gth^2}
\\&
=\P(N^*\le3n/4)
\bigpar{1- e^{-\ccf n^3\gth^2}}
\ge \ccx\cc n^3\gth^2,	\ccdef\cclhr
  \end{split}
\end{equation*}
verifying \eqref{lgw} in this case too (for $c\le\cclgw\cclhr$).

Finally, for $\tau/n\le|\gth|\le\pi$, we obtain by arguing as in
\eqref{jesper}, now using \eqref{fc} and \eqref{ch}, 
\begin{equation*}
  \begin{split}
\bigabs{\gggmn(\eith)}
&
\le e^{-\ccfc n}\P(N^*\le3n/4)+\P(N^*>3n/4)
\le e^{-\ccfc n}+3 e^{-\ccch n}
\\
\le e^{-\cc n},
  \end{split}
\end{equation*}
provided $n$ is large enough.
This completes the proof (for large $n$) for the cases
$\tau/n\le|\gth|\le1/n$ and $1/n\le|\gth|\le\pi$.
\end{proof}

\begin{proof}[Proof of \refT{Tloc}]
Consider any sequence $m=m(n)\ge2$. We will show that \eqref{tloc} holds
uniformly in $k$ for any such sequence $m(n)$; this is equivalent to the
asserted uniform convergence for all $m\ge2$.

Denote the \chf{} of $\gnm$ by $\gfn(\gth)$, and 
 recall that it is given by
  $\gfn(\gth)=\gggmn(\eith)$, see \eqref{pgf}.
It follows from Theorems \refand{Tlimn}{Tlimnm}
 that 
 \begin{equation}\label{sjw}
\frac{\gnm-\mu\nm}{\gs\nm}\dto N(0,1)   
 \end{equation}
as \ntoo.
(To see this we may by considering subsequences assume that $m(n)$ converges
 to either a finite limit or to $\infty$; then \eqref{sjw} is
 \eqref{tlimn1} or \eqref{tlimnm1}.) 
Thus, by the continuity theorem, for any fixed $\gth\in\bbR$,
  \begin{equation}\label{dy}
e^{-\ii\gth\mu\nm/\gs\nm}\gfn(\gth/\gs\nm)\to 	
e^{-\gth^2/2}.
  \end{equation}
Let
  \begin{equation}\label{dy2}
r_n(\gth)\=e^{-\ii\gth\mu\nm/\gs\nm}\gfn(\gth/\gs\nm)\ett{|\gth|\le\pi\gs\nm}-
e^{-\gth^2/2},
  \end{equation}
and note that $r_n(\gth)\to0$ as \ntoo{} for each fixed $\gth$ by \eqref{dy}
since $\gs\nm\to\infty$ by \eqref{tmean2}.

By Fourier inversion we have
\begin{equation*}
  \begin{split}
  \gs\nm&\P(\gnm=k)
=\frac{\gs\nm}{2\pi}\int_{-\pi}^{\pi}e^{-\ii k t}\gf(t)\dd t
\\&
=\frac{1}{2\pi}\int_{-\pi\gs\nm}^{\pi\gs\nm}
 e^{-\ii k\gth/\gs\nm}\gf(\gth/\gs\nm)\dd \gth
\\&
=\frac{1}{2\pi}\intoooo
 e^{\ii(\mu\nm- k)\gth/\gs\nm} \Bigpar{r_n(\gth)+e^{-\gth^2/2}}\dd \gth
\\&
=\frac{1}{2\pi}\intoooo e^{\ii(\mu\nm- k)\gth/\gs\nm} r_n(\gth)\dd \gth
+\frac1{\sqrt{2\pi}} e^{-(\mu\nm- k)^2/2\gss\nm},
  \end{split}
\end{equation*}
and thus, for all $k\in\bbZ$,
\begin{equation*}
  \begin{split}
\lrabs{ \gs\nm\P(\gnm=k)- \frac1{\sqrt{2\pi}} e^{-(\mu\nm- k)^2/2\gss\nm}}
\le\frac{1}{2\pi}\intoooo \bigabs{r_n(\gth)}\dd \gth.
  \end{split}
\end{equation*}
The result \eqref{tloc} follows since
\begin{equation*}
\intoooo \bigabs{ r_n(\gth)}\dd \gth
\to0
\end{equation*}
as \ntoo{} by dominated convergence, using \refL{LGW}; note that if
$|\gth|\le\pi\gs\nm$, then $|\gth|\le n\qqc$ since $\pi^2\gss\nm<n^3$ by
\eqref{tmean2}, and hence \eqref{lgw} yields
\begin{equation*}
  \bigabs{\gfn(\gth/\gs\nm)}
=
  \bigabs{\gggmn(e^{\ii \gth/\gs\nm})}
\le e^{-c n^3\gth^2/\gss\nm}+e^{-cn}
\le e^{-c \gth^2}+e^{-c\gth^{2/3}};
\end{equation*}
hence, for all $n\ge2$ and $\gth\in\bbR$,
\begin{equation*}
  \bigabs{r_n(\gth)}
\le 2e^{-c \gth^2}+e^{-c\gth^{2/3}}.
\end{equation*}
The version with 
$\bmu\nm$ and $\bgss\nm$ follows in exactly the same way, starting with 
 \begin{equation}
\frac{\gnm-\bmu\nm}{\bgs\nm}\dto N(0,1),   
 \end{equation}
which is equivalent to \eqref{sjw} since $\bgss\nm\sim \gss\nm$ and
$\bmu\nm=\mu\nm+o(\gs\nm)$ as \ntoo{} by \refT{Tmean}.
\end{proof}

\newcommand\AAP{\emph{Adv. Appl. Probab.} }
\newcommand\JAP{\emph{J. Appl. Probab.} }
\newcommand\JAMS{\emph{J. \AMS} }
\newcommand\MAMS{\emph{Memoirs \AMS} }
\newcommand\PAMS{\emph{Proc. \AMS} }
\newcommand\TAMS{\emph{Trans. \AMS} }
\newcommand\AnnMS{\emph{Ann. Math. Statist.} }
\newcommand\AnnPr{\emph{Ann. Probab.} }
\newcommand\CPC{\emph{Combin. Probab. Comput.} }
\newcommand\JMAA{\emph{J. Math. Anal. Appl.} }
\newcommand\RSA{\emph{Random Struct. Alg.} }
\newcommand\ZW{\emph{Z. Wahrsch. Verw. Gebiete} }
\newcommand\DMTCS{\jour{Discr. Math. Theor. Comput. Sci.} }

\newcommand\AMS{Amer. Math. Soc.}
\newcommand\Springer{Springer-Verlag}
\newcommand\Wiley{Wiley}

\newcommand\vol{\textbf}
\newcommand\jour{\emph}
\newcommand\book{\emph}
\newcommand\inbook{\emph}
\def\no#1#2,{\unskip#2, no. #1,} 
\newcommand\toappear{\unskip, to appear}

\newcommand\arxiv[1]{\url{arXiv:#1.}}

\def\nobibitem#1\par{}

\end{document}